\documentclass[reqno]{amsart}
\usepackage{graphicx}
\usepackage{amsmath}
\usepackage{mathrsfs}
\usepackage{amssymb}
\usepackage[all]{xy}
\usepackage{color}
\usepackage[colorlinks]{hyperref}

\newtheorem{thm}{Theorem}[section]
\newtheorem{cor}[thm]{Corollary}
\newtheorem{lem}[thm]{Lemma}
\newtheorem{prop}[thm]{Proposition}
\theoremstyle{definition}

\theoremstyle{remark}
\newtheorem{rem}[thm]{Remark}
\numberwithin{equation}{section}

\newcommand{\z}{\mathbb Z}
\newcommand{\q}{\mathbb Q}

\newcommand{\rp}{\mathbb R \mathrm P}

\newcommand{\omsp}{\Omega^{\mathrm{spin}}}

\begin{document}

\title{Free cyclic group actions on highly-connected $2n$-manifolds}
\author{Yang Su}

\address{School of Mathematics and Systems Science, Chinese Academy of Sciences, Beijing 100190,China}
\address{School of Mathematical Sciences, University of Chinese Academy of Sciences, Beijing 100049, China.}
\email{suyang@math.ac.cn}

\author{Jianqiang Yang}

\address{Department of Mathematics, Honghe University, Yunnan 661199, China.}

\email{yangjq\_math@sina.com}




\begin{abstract}
In this paper we study smooth orientation-preserving free actions of the cyclic group $\z/m$ on a class of $(n-1)$-connected $2n$-manifolds, $\sharp g (S^n \times S^n)\sharp \Sigma$, where $\Sigma$ is a homotopy $2n$-sphere. When $n=2$ we obtain a classification up to topological conjugation. When $n=3$ we obtain a classification up to smooth conjugation. When $n \ge 4$ we obtain a classification up to smooth conjugation when the prime factors of $m$ are larger than a constant $C(n)$.
\end{abstract}

\maketitle

\section{Introduction to the main results}
Highly-connected manifolds receive lasting attention from topologists. It was the attempt to understand the topology of $(n-1)$-connected $2n$-manifolds that led to the discovery of the exotic $7$-spheres (\cite{Mil99}). Topological problems of these manifolds are good test cases of theories and techniques of algebraic and differential topology. Wall obtained a classification of $(n-1)$-connected $2n$-manifolds up to almost diffeomorphism (\cite{Wall}). Later Kreck determined the mapping class groups of almost parallelizable $(n-1)$-connected $2n$-manifolds up to extension problems (\cite{Kreck78}). Recent progresses include the study of the extension problems (\cite{Kry02}, \cite{Cr07}, \cite{Kra19}) and the homological stablity of $B\mathrm{Diff}(\sharp g(S^n \times S^n), D)$ (\cite{GaRa}). Other related results are for example the computation of the set of mapping degrees between $(n-1)$-connected $2n$-manifolds (\cite{DW}, \cite{DP}) and the existence of almost complex structures  (\cite{Yang}).

The study of the symmetries of highly-connected manifolds is also an active research field. There is a large volume of research articles discussing the existence of free actions of various finite groups on highly-connected manifolds, see \cite{CrHa} and the bibliography therein for more information. In this paper we focus on free actions of finite cyclic groups $\z/m$ on $\sharp g(S^{n} \times S^{n}) \sharp \Sigma$, where $\Sigma$ is a homotopy $2n$-sphere. When $n$ is odd, $n \not \equiv 1 \pmod 8$  and not equal to $15$, $31$ and possibly $63$, by the classification of  $(n-1)$-connected $2n$-manifolds, and the solution of the Kervaire invariant $1$ problem (Browder, Barratt, Jones, Mahowald, Hill-Hopkins-Ravenel), every  $(n-1)$-connected $2n$-manifold is diffeomorphic to $\sharp g(S^{n} \times S^{n}) \sharp \Sigma$. Our goal is to obtain results on existence, classification and geometric description of these actions.

\medskip

\noindent {\bf Convention.} In this paper all actions are smooth, orientation-preserving and free.

\medskip

We first discuss the existence of free $\z/m$-action on $\sharp g(S^{n} \times S^{n}) \sharp \Sigma$.

\begin{thm}\label{thm:existence}
If there is an orientation-preserving free $\z/m$-action on $\sharp g (S^{n} \times S^{n}) \sharp \Sigma$, then $g+(-1)^n$ is divisible by $m$. For $\sharp g(S^n \times S^n)$ this is also a sufficient condition for the existence of a free $\z/m$-action.
\end{thm}

Next we consider the classification of free $\z/m$-actions on $\sharp g(S^{n} \times S^{n}) \sharp \Sigma$.  When $n=2$, a classification of free $\z/m$-actions on $\sharp g (S^{2} \times S^{2})$ up to topological conjugations can be obtained from the topological classification of $4$-manifolds with finite cyclic fundamental group due to Hambleton-Kreck \cite{HK93}.

\begin{thm}\label{thm:dim2}
Assume that $g+1$ is divisible by $m$. Then up to topological conjugation there is a unique orientation-preserving smooth free $\z/m$-action on $\sharp g(S^2 \times S^2)$.
\end{thm}

 For $n=3$ we have a classification of free $\z/m$-actions on $\sharp g(S^{3}\times S^{3})$ up to  smooth conjugations. Let  $L_m^3=L^{3}(m; 1,1)$ be the standard $3$-dimensional lens space with fundamental group $\z/m$. When $m$ is even, there is a unique orientable $4$-dimensional real vector bundle $\xi$ over $L_{m}^{3}$ with $w_{2}(\xi) \ne 0$. Let $S(\xi)$ be the sphere bundle of $\xi$, its universal cover is $S^3 \times S^3$, since any vector bundle over $S^3$ is trivial.

\begin{thm}\label{thm:dim3}
Assume that $g-1$ is divisible by $m$. When $m$ is odd, up to smooth  conjugation, there is a unique free $\z/m$-action on $\sharp g(S^3 \times S^3)$, with quotient manifold diffeomorphic to \mbox{$(L_m^3 \times S^3 )\sharp \frac{g-1}{m}(S^3 \times S^3)$}. When $m$ is even, there are two conjugacy classes of free $\z/m$-actions on $\sharp g(S^3 \times S^3)$, with quotient manifolds diffeomorphic to $$(L_m^3 \times S^3 )\sharp \frac{g-1}{m}(S^3 \times S^3)\ \ \  \mathrm{or \ \ \ } S(\xi) \sharp \frac{g-1}{m}(S^3 \times S^3).$$
These two manifolds are distinguished by $w_{2}$.
\end{thm}

For $n \ge 4$, we have a classification of  fee $\z/m$-actions on $\sharp g(S^{n} \times S^{n}) \sharp \Sigma$ when the prime factors of $m$ are relatively large. Before stating the result we set up the notations.

 Let $M$ be the quotient space of a free $\z/m$-action on $\sharp g(S^n \times S^n) \sharp \Sigma$, there  is  a canonical identification $\pi_1(M)=\z/m$. From the Leray-Serre spectral sequence of the fibration $\sharp g (S^n \times S^n)\sharp \Sigma \to M \to K(\z/m,1)$ we have identifications $H^{4i}(M)=H^{4i}(K(\z/m,1))=\z/m$ for $i=1, \cdots, [n/4]$ when $n\ne 4k$. When $n=4k$, from the exact sequence
$$0 \to H^n(K(\z/m,1)) \to H^n(M) \to H^n(\sharp_g(S^n \times S^n)\sharp \Sigma )$$
we see that $p_k(M)$ is in the image of $H^n(K(\z/m,1))$. Therefore we may regard the Pontrjagin classes $p_1(M), \cdots p_{[n/4]}(M)$ as in $\z/m$.

For a finite abelian group $A$ of order $|A|$, let $\mathrm{exp}(A)$ denote the set of prime factors of  $|A|$. Let $J_q \colon \pi_q(O) \to \pi_q^S$ be the stable $J$-homomorphism. Let
$$C(n) = \mathrm{max} \{ p \in \bigcup_{q=1}^{n}\mathrm{exp}(\pi_{q}^{S}) \cup \bigcup_{q=n+1}^{2n}\mathrm{exp}(\mathrm{Coker}\,J_{q}) \ | \ p \ge 3, p \ge 2[n/4]-1 \}$$
Let $\Theta_{2n}$ be the group of homotopy spheres of dimension $2n$. By the classical results of Kervaire-Milnor \cite{KerMil}, for $n$  even, $\Theta_{2n}$ is isomorphic to  $\mathrm{Coker}\,J_{2n}$; for $n$ odd, $\Theta_{2n}$ is isomorphic to a subgroup of $\mathrm{Coker}\,J_{2n}$ of index at most $2$. Therefore when the prime factors of $m$ are all larger than $C(n)$, $m$ is invertible in $\Theta_{2n}$.

\begin{thm}\label{thm:gen}
Assume that $g+(-1)^{n}$ is divisible by $m$. If the prime factors of $m$ are all larger than $C(n)$,  then taking the Pontrjagin classes $(p_{1}, \cdots, p_{[n/4]})$ of the quotient manifold gives rise to a one-to-one correspondence between  the set of conjugacy classes of orientation-preserving free $\z/m$-actions on \mbox{$\sharp g (S^{n} \times S^{n}) \sharp \Sigma$} and $(\z/m)^{[n/4]}$.

When $n$ is odd, the quotient manifolds have the form
$$S(\xi) \sharp \frac{g-1}{m} (S^{n} \times S^{n})\sharp \frac{1}{m}\Sigma,$$
where $S(\xi)$ is the sphere bundle of a certain $(n+1)$-dimensional vector bundle $\xi$ over the standard lens space $L^{n}(m;1, \cdots,1)$. When $n$ is even, the quotient manifolds have the form
$$N(\xi) \sharp \frac{g+1}{m}(S^n \times S^n) \sharp \frac{1}{m}\Sigma,$$ where $N(\xi)$ is the result of a surgery on a fiber of $S(\xi)$, with $S(\xi)$  the sphere bundle of a certain $n$-dimensional vector bundle $\xi$ over the standard lens space $L^{n+1}(m;1,\cdots, 1)$.
\end{thm}

\begin{rem}
For more precise description of the quotient manifolds see Section \ref{subsec:model}.
\end{rem}

\begin{rem}
From the knowledge of the stable stems $\pi_{q}^{S}$ (c.~f.~(\cite[p.~376]{Rav}) and the $J$-homomorphisms (\cite{Adams66}) we have a table of $C(n)$ for small $n$.
$$\begin{array}{c|cccccc}
n & 4 & 5 & 6 & 7 & 8 & 9\\
\hline
C(n) & 3 & 3 & 3 & 3 & 5 &5
\end{array}$$
\end{rem}

\begin{rem}
A $\z/m$-action on $\sharp_{g}(S^{n} \times S^{n}) \sharp \Sigma $ induces a $\z[\z/m]$-module structure on the homology group  $H_{n}(\sharp  {g}(S^{n} \times S^{n}) \sharp \Sigma) = \z^{2g}$. In contrast to the abundance of integral representations of $\z/m$ (\cite{Reiner}), Theorem \ref{thm:dim3} and Theorem \ref{thm:gen} imply that, when the prime factors of $m$ are larger than $C(n)$,  the only modules which  can be realized by free actions are
$$\z^{2}\oplus \z[\z/m]^{2r} \textrm{\ \ ($n$ odd)}\ ,  \ I^{2} \oplus \z[\z/m]^{2r} \textrm{\ \ ($n$ even)}$$
where $I$ is the augmentation ideal. 
\end{rem}

\begin{cor}\label{cor:l}
Let $L$ and $L'$ be $3$-dimensional lens spaces of fundamental group $\z/m$. Then $L \times S^3$ is diffeomorphic to $L' \times S^3$.  In general, let $L$ and $L'$ be $n$-dimensional lens spaces with fundamental group $\z/m$. If $L$ and $L'$ have equal Pontrjagin classes (under certain identification of their fundamental groups), and the prime factors of $m$ are larger than $C(n)$, then $L \times S^n$ is diffeomorphic to $L' \times S^n$.
\end{cor}

\begin{rem}
By a classical result of Mazur \cite{Ma61}, if $L$ and $L'$ are stably tangential homotopy equivalent $n$-manifolds, then $L \times S^n \times \mathbb R$ is diffeomorphic to $L' \times S^n \times \mathbb R$. For spherical space forms of odd order fundamental group, Pontrjagin classes determine the stable tangent bundle. Thus Corollary \ref{cor:l} is an improvement of the classical result under the condition of the prime factors of $m$.
\end{rem}

\begin{rem}
The same method of this paper can be used to study free $\z/m$-action on almost parallalizable $(n-1)$-connected $2n$-manifolds with Kervaire-Arf invariant $1$. These manifolds exist in dimension $30$, $62$ and possibly $126$ and have the form $K^{2n} \sharp g(S^n \times S^n)$, where $K$ is a smooth Kervaire manifold. Similar to Theorem \ref{thm:existence}, a necessary condition for  the existence of a free $\z/m$-action is that $g-1$ is divisible by $m$. It's shown in \cite[Theorem 9.1]{CrHa} that there is a free $\z/m$-action on some smooth Kervaire manifold $K_0$. An arbitrary  smooth Kervaire manifold has the form $K= K_0 \sharp \Sigma$. When the prime factors of $m$ are larger than $\mathrm{exp}(\mathrm{Coker}\,J_{2n})$, by taking equivariant connected sum with $\frac{1}{m}\Sigma$ and copies of $S^{n} \times S^{n}$ one constructs a free $\z/m$-action on $K\sharp g(S^{n} \times S^{n})$.  The method of this paper applies to the classification of these actions. When the prime factors of $m$ are larger than $C(n)$, the only remaining work is the analysis of the surgery obstruction, i.~e.~an algebraic result similar to Lemma \ref{lem:s}.
\end{rem}

In the remaining part of the introduction we describe the strategy of the classification of actions used in this paper. The classification of free group actions on simply-connected manifolds up to conjugation is equivalent to the classification up to diffeomorphism of the quotient manifolds with a fixed identification of the fundamental group with the given group. We will apply the modified surgery method to attack this classification problem. This usually consists of  three steps (for more details see \cite{Kreck99}):

\begin{enumerate}
\item Find a fiberation $p \colon B \to BO$ such that the manifolds under consideration have liftings $\overline \nu$ of their normal Gauss map $\nu$, and $\overline \nu$ is a homotopy equivalence up to the middle dimension. Typically one uses the so-called normal $q$-type $B^q(M)$, which is canonically associated to a manifold $M$, and  indeed is the $q$-stage Moore-Postnikov tower of the normal Gauss map $\nu \colon M \to BO$.

\item Compute the normal bordism group $\Omega_m(B,p)$. There is an obvious action of the group of fiber-homotopy self-equivalences $Aut(p)$ on $\Omega_m(B,p)$ . A manifold $M$ corresponds to an orbit of this action.

\item When two manifolds are cobordant in $\Omega_m(B,p)$, there is a surgery obstruction  $\theta \in l_{m+1}(\pi_1, w_1)$ to transforming by surgery a cobordism to an $s$-cobordism. Study this surgery obstruction.
\end{enumerate}

\medskip

\noindent {\bf Structure of the paper.} We will prove Theorem \ref{thm:existence} and Theorem \ref{thm:dim2} in \S \ref{sec:2}. For the classification of actions, to illustrate the method,  we first study the case of $n=3$: in \S \ref{subsec:norm} we determine the normal $2$-type and  the normal bordism groups; in \S \ref{sec:ob} we analyze the surgery obstruction and prove Theorem \ref{thm:dim3}. The study of higher dimensional cases goes along the same line, but technically more complicated. This is the content of \S \ref{sec:high}. A computation of the bordism group with the Atiyah-Hirzebruch spectral sequence is in the Appendix.

\section{Proofs of Theorem \ref{thm:existence} and Theorem \ref{thm:dim2}}\label{sec:2}
The proof of  Theorem \ref{thm:existence} only uses elementary algebraic topology.

\begin{proof}
Assume that there is an orientation-preserving free $\z/m$-action on $\sharp {g}(S^{n} \times S^{n}) \sharp \Sigma$, with quotient manifold $M$. The Euler characteristic of $\sharp g (S^{n} \times S^{n}) \sharp \Sigma $ is $2(1+(-1)^{n}g)$. Therefore the Euler characteristic of $M$ is $2(1+(-1)^{n}g)/m$. When $n$ is even, the signature of $\sharp g (S^{n} \times S^{n}) \sharp \Sigma$ is $0$, therefore the signature of $M$ is also $0$. This implies that $\dim_{\q}H_{n}(M;\q)$ is even; when $n$ is odd, the intersection form on $H_{n}(M)$ is skew-symmetric, therefore $\dim_{\q}H_{n}(M;\q)$ is even. Thus in both cases $\chi(M)$ is even. This shows that $g+(-1)^{n}$ is divisible by $m$.

Now we construct free $\z/m$-actions on $\sharp g(S^n \times S^n)$ when $g+(-1)^{n}$ is divisible by $m$. For $n$ odd, let $L_m^n$ be an $n$-dimensional lens space with fundamental $\z/m$, then the universal cover of $(L^{n}_{m}\times S^{n} )\sharp \frac{g-1}{m} (S^{n} \times S^{n})$ is $\sharp {g}(S^{n} \times S^{n})$. The deck transformation is an orientation-preserving free $\z/m$-action.

For $n$ even, let $L_m^{n+1}$ be an $(n+1)$-dimensional lens space with fundamental group $\z/m$, consider $L_{m}^{n+1} \times S^{n-1}$, the universal cover is $S^{n+1} \times S^{n-1}$, with the deck transformation as a free $\z/m$-action. Choose an appropriate equivariant embedding $\cup_{i=1}^{m} D^{n+1} \times S^{n-1} \hookrightarrow S^{n+1} \times S^{n-1}$, and do surgery. Surgery on one copy of $D^{n+1} \times S^{n-1}$ produces $S^{2n}$. Then surgery on the remaining $m-1$ copies of $D^{n+1} \times S^{n-1}$ embedded in $S^{2n}$ produces $\sharp_{m-1} (S^{n} \times S^{n})$. This gives an orientation-preserving free $\z/m$-action on $\sharp_{m-1}(S^{n} \times S^{n})$. Making equivariant connected sum with copies of $S^{n} \times S^{n}$'s gives rise to a free $\z/m$-actions on $\sharp {g}(S^{n} \times S^{n})$.
\end{proof}

Theorem \ref{thm:dim2} is a consequence of the topological classification of $4$-manifolds with finite cyclic fundamental group.

\begin{thm}\cite[Theorem C]{HK93}
Let $M$ be a closed, oriented $4$-manifold with finite cyclic fundamental group. Then $M$ is classified up to homeomorphism by the fundamental group, the intersection form on $H_2(M;\z)/\mathrm{Tors}$, the $w_2$-type, and the Kirby-Siebenmann invariant. Moreover, any isometry of the intersection form can be realized by a homeomorphism.
\end{thm}
\begin{proof}[Proof of Theorem \ref{thm:dim2}]
Let $M$ be the quotient manifold of a free $\z/m$-action on $\sharp_g(S^2\times S^2)$. When $m$ is odd clearly $w_2(M)=0$. When $m$ is even, since $\widetilde M$ is spin, $w_2(M)$ is trivial on integral homology classes. The Wu formula implies $w_2(M)^2 = w_4(M) \equiv \chi(M) \equiv 0\pmod 2$ since $\chi(M)$ is even. Therefore $w_2(M)=0$. The signature of $M$ is $0$. Hence the intersection form on $H_2(M)/\mathrm{Tors}$ is even and indefinite, thus is determined by its signature and rank, where $\mathrm{rank}H_2(M)=2(g+1)/m-2$. The Kirby-Siebenmann invariant of $M$ is clearly $0$. Therefore by the above theorem all such quotient manifolds are homeomorphic, and a homeomorphism can be chosen to preserve the identification of the fundamental groups with $\z/m$. Hence there is a unique action up to topological conjugation.
\end{proof}

\section{The $n=3$ case}\label{sec:n3}
\subsection{The normal $2$-type and normal bordism group}\label{subsec:norm}
Let $M^6$ be the quotient of an orientation-preserving free $\z/m$-action on $\sharp_g (S^3 \times S^3)$. Then $\pi_1(M) = \z/m$, $\pi_2(M)=0$; When $m$ is odd, $M$ is spin; when $m$ is even, $M$ may be spin or non-spin.

When $M$ is spin, let
$$p \colon B=K(\z/m, 1) \times B\mathrm{Spin} \stackrel{c \times p_0}{\longrightarrow} BO \times BO \stackrel{\oplus}{\longrightarrow} BO$$
where $c$ is the trivial map, and $p_0$ is the canonical projection. Then $(B,p)$ is the normal $2$-type of $M$, and a normal $2$-smoothing $\bar \nu \colon M \to B$ is given by the classifying map $M \to K(\z/m,1)$ of $\pi_1(M)=\z/m$ and a spin structure of $M$.

When $M$ is non-spin (in this case $m$ must be even), let $\gamma$ be a complex line bundle over $K(\z/m,1)$ with $c_1(\gamma)=1 \in H^2(K(\z/m,1)) = \z/m$, then $w_1(\gamma)=0$, $w_2(\gamma)$ is the non-trivial element in $H^2(K(\z/m,1);\z/2) = \z/2$. Let $$p \colon B=K(\z/m, 1) \times B\mathrm{Spin} \stackrel{\gamma \times p_0}{\longrightarrow} BO \times BO \stackrel{\oplus}{\longrightarrow} BO$$
Then $(B,p)$ is the normal $2$-type of $M$, a normal $2$-smoothing $\bar \nu \colon M \to B$ is given by the classifying map $ f \colon M \to K(\z/m,1)$ of $\pi_1$ and a spin structure on $f^*\gamma \oplus \nu M$.

In the spin case the normal bordism group $\Omega_n(B,p)$ is identified with  the spin bordism group $\omsp_n(K(\z/m,1))$; in the non-spin case the normal bordism group is identified with the twisted spin bordism group $\omsp_n(K(\z/m,1);\gamma)$. We may use the Atiyah-Hirzebruch spectral sequence to compute these groups, the details are in the Appendix.

\begin{prop}\label{prop:bordism}
The normal bordism group $\Omega_{6}(B,p)=0$.
\end{prop}

\subsection{Model manifolds}
Let $L_{m}^{3}=L^{3}(m;1, 1)$ be the standard lens space with fundamental group $\z/m$. When $m$ is even, there is a unique orientable $4$-dimensional vector bundle $\xi$ over  $L_{m}^{3}$ with non-trivial $w_{2}$. Let $S(\xi)$ be the corresponding sphere bundle, then $\pi_{1}S(\xi) = \pi_{1}(L_{m}^{3})=\z/m$, and $w_{2}(S(\xi)) \ne 0$. The universal cover of $S(\xi)$ is  diffeomorphic to $S^{3} \times S^{3}$. Let $M_0$ be
$$(L_m^3 \times S^3 )\sharp \frac{g-1}{m}(S^3 \times S^3) \  \ \mathrm{or}  \ \ S(\xi) \sharp \frac{g-1}{m}(S^{3} \times S^{3})$$
as the model manifold in Theorem \ref{thm:dim3}, $M$ be the quotient manifold of an orientation-preserving free $\z/m$-action on $\sharp g (S^3 \times S^3)$. We are going to show that there is a diffeomorphism between $M$ and $M_0$, which preserves the given identification of $\pi_1$ with $\z/m$. Therefore the $\z/m$-action is conjugate to the standard $\z/m$-action on $\sharp g(S^3 \times S^3)$.

\subsection{The surgery obstruction}\label{sec:ob}
By Proposition \ref{prop:bordism} $M_0$ and $M$ are normally bordant in the normal $2$-type $B$. Let $(W, \bar \nu)$ be a normal bordism between them,  by \cite[Proposition 6]{Kreck99} there is a  surgery obstruction $\theta(W, \bar \nu) \in l^{\sim}_7(\z/m, +)$ to transforming $W$ to an $s$-cobordism by surgery. (It's easy to check in this case $\langle w_{4}(B), \pi_{4}(B)\rangle \ne 0$, hence the obstruction is in the tilde $l$-monoid.) In the sequel denote the integral group ring  $ \z[\z/m]$ by $\Lambda$. Elements in the  abelian monoid $l^{\sim}_7(\z/m, +)$ are equivalence classes $[H_-^r(\Lambda), V]$, where  $H_-^r(\Lambda)$ is the skew-symmetric hyperbolic form of rank $r$ over $\Lambda$, and $V$ is a based half rank direct summand in $H_-^{r}(\Lambda)$. There is an abelian subgroup $L^{\sim}_7(\z/m,+) \subset l^{\sim}_7(\z/m,+)$, consisting of those $[H_-^r(\Lambda), V]$ in which $V$ is a Lagrangian.

From the geometric construction of $M_{0}$ we have
$$\mathrm{Ker}(\pi_3(M_0) \to \pi_3(B))/\mathrm{rad} = \z^2 \perp H_-^r(\Lambda)$$
with $r=(g-1)/m$. Here $\z^2$ is a trivial $\Lambda$-module, with a hermitian form $(\lambda_{0}, \widetilde \mu_{0})$. On the standard basis $\{u_{1}, u_{2}\}$,  $\lambda_0(u_1, u_2)=s=1+g+g^2+\cdots +g^{m-1} \in \Lambda$ is the norm element, $\lambda_0(u_i, u_i)=0$ and $\widetilde \mu_0(u_i)=0$ for  $i=1,2$. The surgery obstruction $\theta(W, \bar \nu)$ can be described by the following lemma.

\begin{lem}\label{lem:surj}
If $S_{1} \to \z^{2}$ and $S_{2} \to H_{-}^{r}(\Lambda)$ are surjections from free based modules $S_{i}$, with induced hermitian forms,  then there are half rank embeddings $S_{i} \subset H_{-}^{t_{i}}(\Lambda)$,  such that $\theta(W, \bar \nu)=[H_-^{t_{1}}(\Lambda), S_{1}] + [H_-^{t_{2}}(\Lambda), S_{2}] + \rho$, with $\rho \in L_7^{\sim}(\z/m)$.
\end{lem}
\begin{proof}
This is from \cite[Proposition 2 (3)]{HKT94} (see also \cite[Proposition 8 (ii)]{Kreck99}).
\end{proof}

Here we may simply take the obvious  $S_{1}=\Lambda^2 \to \z^{2}$ and $S_{2}= H_-^r(\Lambda) \stackrel{\mathrm{id}}{\rightarrow} H_-^r(\Lambda)$, where on the standard basis $\{ v_{1}, v_{2}\}$ of $S_{1}$ the induced hermitian form is $\lambda(v_{1}, v_{2})=s$, $\lambda(v_{i}, v_{i})=0$ and $\widetilde \mu(v_i)=0$. Then a further simplification of the surgery obstruction is based on the following lemma.

\begin{lem}\label{lem:lag} \cite[Proposition 8 (iii)]{Kreck99}
If $S_{i}$ has a Lagrangian complement in $H_-^{t_{i}}(\Lambda)$, then $[H_-^{t_{i}}(\Lambda), S_{i}]+\rho'_{i}$ is elementary for some $\rho'_{i} \in L_7^{\sim}(\z/m)$.
\end{lem}

Clearly $S_{2}$ has a Lagrangian complement, since any isometric embedding of $H_{-}^{r}(\Lambda)$ into $H_{-}^{2r}(\Lambda)$ can be moved to the standard position by an isometry and hence has a Lagrangian complement.  For $S_{1}$ we have the following

\begin{lem}\label{lem:s}
Let $S=\mathrm{span}_{\Lambda}(v_{1}, v_{2})$ be a rank $2$ based free $\Lambda$-module with a skew hermitian form $(\lambda, \widetilde \mu_{0})$, given by $\lambda(v_{1}, v_{2})=s$, $\lambda(v_{i}, v_{i})=0$ and $\widetilde{\mu}(v_{i})=0$. Then for any isometric embedding of $S$ into $H_-^2(\Lambda)$ as a half rank direct summand there is a Lagrangian complement.
\end{lem}

We postpone the proof of this technical lemma to the end of this section.  By the above lemmas we have shown that the surgery obstruction $\theta(W, \bar \nu)$ is essentially in $L^{\sim}_7(\z/m,+)$, so we now analyze this group. There is an exact sequence
$$0 \to L^{s \sim}_7(\z/m, +) \to L^{\sim}_7(\z/m,+) \stackrel{\tau}{\rightarrow} Wh(\z/m)$$
where $\tau$ is the torsion of the matrix mapping the given basis of $V$ together with the induced basis of $V^{\perp}$ to the standard basis of $H_-^r(\Lambda)$. Note that $\tau$ factors through $\tau \colon U(\Lambda) \to Wh(\z/m)$, where $U(\Lambda)$ is the stable unitary group over $\Lambda$. For any matrix $A \in U(\Lambda)$, we have $A J A^{\dag}=J$, where $A^{\dag}$ is the transpose conjugate of $A$, and $$J=\left ( \begin{array}{cc}
0 & I \\
-I & 0 \end{array} \right )$$
Since $\tau (A^{\dag})=\tau(A)$ (\cite[Lemma 6.7]{Mil66}) we see that $\tau(A)$ is a $2$-torsion. But $Wh(\z/m)$ is a free abelian group \cite{Mil66}, therefore $\tau$ is trivial and $L^{s \sim}_{7}(\z/m,+) \cong L^{\sim}_{7}(\z/m,+)$.

The following lemma due to M.~Kreck relates the tilde $L$-group to the ordinary $L$-group. It might have independent interest in the classification of $6$- and $14$-manifolds via modified surgery, e.~g.~the classification of certain $6$-manifolds with surface fundamental groups in \cite{HK16}.

\begin{lem}[Kreck]\label{lem:kr}
For any group $\pi$ there is a surjection  $L^{s}_{3}( \pi) \to L^{s \sim }_{3}(\pi)$.
\end{lem}
\begin{proof}
The groups $L^{s \sim }_{*} ( \pi) $ and $L^{s}_{*}( \pi)$  are $L$-groups of quadratic forms with form parameter $\Gamma=\{ 1, a+\bar a \ | \ a \in \z\pi \}$ and $\Lambda= \{ a+\bar a \ | \ a \in \z\pi \}$  respectively. By \cite[Theorem 11.2]{Bak}  there is an exact sequence
$$L^{s}_{3}(\pi) \to L^{s \sim}_{3}( \pi) \to S(\Gamma/\Lambda) \to L^{s}_{2}( \pi) \to L^{s \sim }_{2}(\pi) \to 0$$
with $S(\Gamma/\Lambda)=\z/2$.
By the commutative diagram
$$\xymatrix{
\z/2 \ar [r] & L^{ s}_{2}(\pi)  \ar[r] & L^{s \sim}_{2}(\pi)  \\
\z/2 \ar[r] \ar[u]^{=} & L^{ s}_{2}(1)  \ar[u] \ar[r] & L^{s \sim}_{2}(1 ) \ar[u] }
$$
and the facts that $L^{s}_{2}(1) \cong \z/2$, $L^{s \sim }_{2}(1)=L^{2}(\z)=0$ we see that $L^{s}_{3}( \pi) \to L^{s \sim }_{3}(\pi)$ is surjective.
\end{proof}

The $L^s_*$-groups of $\mathbb Z/m$ are known (c.~f.\cite[\S 5.4]{Wall76} and \cite[p.227]{HT99}):  when $m$ is odd $ L^s_7(\z/m)=0$;  when $m$ is even $L^s_7(\mathbb Z/m) \cong \mathbb Z/2$, detected by a codimension $1$ Arf invariant. Therefore we have shown

\begin{prop}\label{prop:L}
When $m$ is odd $L^{\sim}_7(\z/m)=0$; when $m$ is even $L^{\sim }_7(\mathbb Z/m)$ is either trivial or isomorphic to $\mathbb Z/2$, in the latter case it is detected by a codimension $1$ Arf invariant.
\end{prop}

\begin{proof}[Proof of Theorem \ref{thm:dim3}]
For a normal $B$-bordism $W$ between $M_0$ and $M$, by Lemma \ref{lem:surj}, Lemma \ref{lem:lag}, Proposition \ref{prop:L}, and the fact that there is a closed $7$-manifold with $B$-structure and non-trivial codimension $1$ Arf invariant, by taking the disjoint union with this manifold, we have a $B$-bordism  $(W, \bar \nu)$  with elementary surgery obstruction $\theta(W, \bar \nu)$. Also $M$ and $M_{0}$ have equal Euler characteristic, therefore $M$ is diffeomorphic to $M_{0}$ by a diffeomorphism which  induces the given identification of $\pi_{1}$.
\end{proof}

For the proof of Lemma \ref{lem:s} we need the following auxiliary lemma.

\begin{lem}\label{lem:module}
Let $s=1+g+g^{2}+\cdots + g^{m-1}$ be the norm element in the group ring $\Lambda=\z[\z/m]$, $A \subset \Lambda$ be an ideal such that $A+(s)=\Lambda$. Then $A=u \cdot \Lambda$ where $u=1+g+g^{2}+\cdots +g^{l-1}$ for some positive integer $l$ coprime  to $m$.
\end{lem}
\begin{proof}
Let  $\varepsilon \colon \Lambda \to \mathbb Z$, $\sum_g a_g g \mapsto \sum_g a_g$ be the augmentation, $I =\mathrm{Ker}(\varepsilon)$ be the augmentation ideal. Since $A + (s) = \Lambda$, there exist $\alpha \in A$ and $k \in \z$ such that $1=\alpha+k \cdot s$. For any $x \in I$, we have $x=\alpha\cdot x + k\cdot s\cdot x=\alpha \cdot x$, since $x \cdot s = \varepsilon (x) \cdot s =0$. Therefore $ I \subset A$.

Assume $\varepsilon(A)= l \cdot \z$, $l>0$. Since $\varepsilon (s)=m$ and $A  +(s) = \Lambda$,  we have $(l,m)=1$. From the commutative diagram
$$\xymatrix{
0 \ar[r] & I \ar[d]^{=} \ar[r] & A \ar[d] \ar[r]^{\varepsilon} & l \cdot \z \ar[d]\ar[r] & 0 \\
0 \ar[r] & I  \ar[r] & \Lambda  \ar[r]^{\varepsilon} & \z \ar[r] & 0 }$$
we get an isomorphism $\bar \varepsilon \colon \Lambda / A \stackrel{\cong}{\rightarrow} \z/l$. But $\Lambda /A = (s)/A\cap (s)$, therefore
 $A \cap (s) =l \cdot (s )$. From $l=l\alpha+lk\cdot s$ and $lk\cdot s \in l \cdot (s)$ we see that $l \cdot \Lambda \subset A$.  Therefore $l \cdot \Lambda + I \subset A$. On the other hand, for any $x \in A$ we have $x = x-\varepsilon(x) + \varepsilon(x) \in I + l \cdot \Lambda$. Therefore $A = l \cdot \Lambda + I$.

 Since $l$ and $m$ are coprime, there are integers $a$ and $b$ with $b > 0$, such that $am-bl=1$. Let $u=1+g+\cdots + g^{l-1}$, $v=-g(1+g^l+\cdots +g^{(b-1)l})$, then
$$u \cdot v=\frac{1-g^{l}}{1-g} \cdot (-g)\frac{1-g^{bl}}{1-g^{l}}=\frac{g^{bl+1}-g}{1-g}=1-\frac{1-g^{am}}{1-g}=1-a \cdot s$$
Notice that  $\varepsilon(u)=l$, from the equations $l=(l-u)+u \in I + u \cdot \Lambda$ and $u=(u-l)+l \in I + l \cdot \Lambda$ we see that $l \cdot \Lambda + I = u \cdot \Lambda +I$. From $uv \cdot I =(1-as)\cdot I=I$ we see that $I \subset u \cdot \Lambda$. Therefore we have $A= l \cdot \Lambda + I = u \cdot \Lambda + I = u \cdot \Lambda$.
\end{proof}

\begin{proof}[Proof of Lemma \ref{lem:s}]
We use the same idea as in the proof of \cite[section 4]{HKT94}: by a sequences of isometries of the ambient space we move the image of $S$ to a position where we can read off a Lagrangian complement.

Let $\{e_{1}, e_{2}, f_{1}, f_{2} \}$ be the standard symplectic basis of $H_-^{2}(\Lambda)$. For convenience we will not distinguish elements in $S$ and their image under the embedding.

We first consider the case when $m$ is odd, which is simpler. First of all, since $v_{1}$ is a primitive vector and $\widetilde{\mu}(v_{1})=0$, after an isometry of $H_{-}^{2}(\Lambda)$, we may assume that $v_{1}=e_{1}$, then we may assume
$$v_{2}=a_{1}\cdot e_{1} + a_{2} \cdot e_{2} + s \cdot f_{1} + b_{2} \cdot  f_{2}, \ \ \ a_{i}, b_{i} \in \Lambda$$
where the coefficient of $f_{1}$ is determined by the condition $\lambda(v_{1}, v_{2})=s$.
Note that the $a_{2} \cdot e_{2} + s \cdot f_{1} + b_{2} \cdot f_{2}$ is a primitive vector, therefore $a_{2} \cdot \Lambda + b_{2} \cdot \Lambda + s \cdot \Lambda = \Lambda$. From Lemma \ref{lem:module} we see that $a_{2} \cdot \Lambda + b_{2} \cdot \Lambda = u \cdot \Lambda$ with $u= 1+g+g^2 +\cdots +g^{l-1}$, and $u v + as =1$ for some $v \in \Lambda$,
$a\in \z$.  Let $a_{2} = u \cdot \alpha$, $b_{2}=u \cdot \beta$. Now consider $\alpha \cdot e_{2} + \beta \cdot f_{2}$, this is a primitive vector with $\widetilde{\mu}(\alpha \cdot e_{2} + \beta \cdot  f_{2}) = [\alpha \bar \beta] \in \Lambda/\langle 1, x+\bar x\rangle$. From the equation
$$0 =\lambda(v_{2}, v_{2})=u\bar u (\alpha \bar \beta - \bar \alpha \beta)$$
and the fact that $u$ is not a zero-divisior, we see that $ \alpha \bar \beta - \bar \alpha \beta=0$. This implies that $ [\alpha \bar \beta]=0 \in \Lambda/\langle 1, x+\bar x\rangle$, since $m$ is odd. On the other hand, $v \cdot e_{2} + s \cdot f_{2}$ is also a primitive vector with $\widetilde{\mu}(v \cdot e_{2} + s \cdot f_{2}) = [v \bar s] = [\varepsilon(v) s]=0 \in \Lambda/\langle 1, x+\bar x\rangle$, where $\varepsilon$ is the augmentation. Therefore there is an isometry of $\mathrm{span}_{\Lambda}(e_{2}, f_{2})$, sending $\alpha \cdot e_{2} + \beta \cdot f_{2}$ to $v \cdot e_{2} + s \cdot f_{2}$. Now we may assume
$$v_{1}= e_{1}, \ \ \ v_{2}= a_{1} \cdot e_{1} + uv \cdot e_{2} + s \cdot f_{1} + us \cdot f_{2}$$
Let $w_{1}= -a \cdot e_{2}+ f_{1}$, $w_{2}= -a \cdot e_{1} + f_{2}$. It's easy to see that the submodule $U$ generated by $w_{1}$ and $w_{2}$ has $\lambda|_{U \times U} =0$ and $\widetilde{\mu} |_{U} =0$, and the matrix
$$\left ( \begin{array}{cccc}
1 & a_{1} & 0   & -a \\
0 & uv      & -a &  0   \\
0 & s         & 1    & 0 \\
0 & us       & 0    & 1 \end{array} \right )$$
has determinant $=uv+as =1$. Therefore $U$ is a Lagrangian complement of $S$.

When $m=2m'$ is even, we replace $v_2$ by $v_1+v_2$, then we have $\lambda(v_1, v_2)=s$, $\lambda(v_i, v_i)=0$ for $i=1,2$ and $\widetilde \mu(v_1)=0$, $\widetilde \mu(v_2)= [g^{m'}]$. As before we may first assume that $v_1=e_1$ and $v_{2}=a_{1}\cdot e_{1} + a_{2} \cdot e_{2} + s \cdot f_{1} + b_{2} \cdot  f_{2}$. Now we need an additional transformation to control the coefficient of $e_1$: since $a_{2} \cdot e_{2} + s \cdot f_{1} + b_{2} \cdot  f_{2}$ is a primitive vector, there exist $r,k,t \in \Lambda$ such that $r a_2 + k s + t b_2=-a_1$. Apply the isometies $R$ and $T$ given by the matrices (indeed these are transvections by $(e_1, -rf_2)$ and $(e_2, tf_1)$ respectively)
$$
R=\left ( \begin{array}{cccc}
1 & r & 0 & 0\\
0 & 1 & 0 & 0\\
0 & 0 & 1 & 0\\
0 & 0 & -\overline r & 1
\end{array} \right ), \ \ T = \left ( \begin{array}{cccc}
1 & 0 & 0 & t\\
0 & 1 & -\overline t & 0\\
0 & 0 & 1 & 0\\
0 & 0 & 0 & 1
\end{array} \right )$$
the coefficient of $e_1$ in the vector $R \circ T(v_2)$ now equals $hs$ for some $h \in \z$. Therefore we have
$$v_1 = e_1, \ \ \ v_2=hs \cdot e_1 + a_2 \cdot e_2 + s \cdot f_1 + b_2 \cdot f_2$$
Now we proceed as before, assume $a_2= u \cdot \alpha$, $b_2 = u \cdot \beta$, with $u=1+g+ \cdots g^{l-1}$, $v=1+g^l + \cdots +g^{(b-1)l}$, $uv + as =1$, $bl+am =1$.  The primitive vector $v \cdot e_2 + s \cdot f_2$ has $$\widetilde \mu (v \cdot e_2 + s \cdot f_2) = [vs]=[\varepsilon (v) s] = [g^{m'}] \in \Lambda/\langle 1, x+\bar x \rangle$$
since $\varepsilon (v) =b$ is odd. The primitive vector $\alpha \cdot e_2 + \beta \cdot f_2$ has $\widetilde \mu (\alpha \cdot e_2 + \beta \cdot f_2) = [\alpha \bar \beta]$. Write $\alpha \bar \beta = \sum_i c_i g^i$, then by the equation $\alpha \bar \beta = \bar \alpha \beta$ (since $\lambda(v_2, v_2)=0$) we have $\widetilde \mu (\alpha \cdot e_2 + \beta \cdot f_2)=[c_{m'}g^{m'}] \in \Lambda/\langle 1, x+\bar x \rangle$. From the equation
$$[g^{m'}]= \widetilde \mu (v_2)= [hs\cdot s]+[u \bar u \alpha \bar \beta]=[hm \cdot s]+[u \bar u \alpha \bar \beta]= [u \bar u \alpha \bar \beta]$$
(since $m$ is even), and the equation
$$[u \bar u \alpha \bar \beta]=[(1+g+ \cdots +g^{(l-1)})(1+g^{-1}+ \cdots + g^{-(l-1)})\cdot \sum_{i=0}^{m-1} c_i g^i]=[lc_{m'}g^{m'}]=[c_{m'}g^{m'}]$$
(since $l$ is odd) we see that $c_{m'}$ is odd, therefore $\widetilde \mu (\alpha \cdot e_2 + \beta \cdot f_2)=[g^{m'}]=\widetilde \mu (v \cdot e_2 + s \cdot f_2)$. Now there is an isometry of $\mathrm{span}_{\Lambda}(e_{2}, f_{2})$, sending $\alpha \cdot e_{2} + \beta \cdot f_{2}$ to $v \cdot e_{2} + s \cdot f_{2}$. The rest is as before.
\end{proof}

\begin{rem}
In the decomposition $\mathrm{Ker}(\pi_3(M_0) \to \pi_3(B))/\mathrm{rad} = \z^2 \perp H_-^r(\Lambda)$ with $r=(g-1)/m$, if $r \ge 2$, Theorem 5 of \cite{Kreck99} implies that the surgery obstruction $\theta(W,\bar \nu)$ is elementary.
\end{rem}

\section{Higher dimensional case}\label{sec:high}
In this section we study free $\z/m$-action on $\sharp g(S^n \times S^n) \sharp \Sigma$ for $n\ge 4$ and prove Theorem \ref{thm:gen}. This section is organized as follows: first, for a given tuple $(p_1, \cdots, p_k) \in (\z/m)^{[n/4]}$, we construct a free $\z/m$-action on $\sharp g(S^n \times S^n) \sharp \Sigma$ with quotient manifold $M_0$, such that $p_i(M_0)=p_i$ for $i=1, \cdots, [n/4]$. This $M_0$ serves as a model manifold, i.~e., for any quotient manifold $M$ which has the same Pontrjagin classes $p_i(M)$ for $i=1,\cdots, [n/4]$, we will show that $M_0$ and $M$ are bordant in a common normal $(n-1)$-type with the given identification of $\pi_1$. Then we show that the surgery obstruction is elementary. Therefore there is a diffeomorphism between $M_0$ and $M$ inducing the given identification of $\pi_1$, hence the two actions are conjugate. To keep the notations simple we will only deal with the case of $\sharp {g}(S^{n} \times S^{n})$. The proof for the general situation is similar.
Recall that we assume the prime factors of $m$ are larger than $C(n)$. Denote $k=[n/4]$.

\subsection{Model manifolds}\label{subsec:model}
In this subsection we construct model manifolds $M_{0}$ with fundamental group $\z/m$, universal cover $\sharp {g}(S^{n} \times S^{n})$ and given Pontrjagin classes
$$(p_{1}(M_{0}), \cdots, p_{k}(M_{0})) \in (\z/m)^{k}.$$

\subsubsection{$n$ odd}
\begin{lem}\label{lem:odd}
Let $L_{m}^{n}=L_{m}^{n}(1,\cdots, 1)$ be the standard lens space, $g \colon S^{n} \to L_{m}^{n}$ be the covering map. Given $a_i \in H^{4i}(L_m^n)=\z/m$ for $i=1,\cdots , k$, there is a stable vector bundle $\xi$ over $L_m^n$ such that
\begin{enumerate}
\item $p_i(\xi)=a_i$ for $i=1, \cdots, k$;
\item $g^{*}\xi$ is trivial.
\end{enumerate}
\end{lem}
\begin{proof}
Let $\mathrm{Pon} \colon BO \to K(\z, 4) \times \cdots \times K(\z, 4k)$ be the map given by the universal Pontrjagin class $p_1, \cdots p_k$,  then the image of the induced map $\pi_{4i}(BO)=\z \to \pi_{4i}(K(\z, 4i))=\z$ has index $a_i \cdot (2i-1)!$, where $a_{i}=1$ when $i$ is even and $a_{i}=2$ when $i$ is odd (\cite{Kervaire}). Denote the homotopy fiber by $F$.

For $a=(a_{1}, \cdots , a_{k})$,
let $h_{a} \colon L_{m}^{n} \to K(\z, 4) \times \cdots \times K(\z, 4k)$ be the classifying map of $a$, consider the following lifting problem
$$\xymatrix{
& BO \ar[d]^{\mathrm{Pon}} \\
L_{m}^{n} \ar[ur]^{\overline h_{a}} \ar[r]^{h_{a}\ \ \ \ \ \ \ \ \ \ \ \ } &  K(\z, 4) \times \cdots \times K(\z, 4k)}$$
The obstructions to a lifting $\overline{h}_{a}$ lie in $H^i(L_{m}^{n};\pi_{i-1}(F))$. When the prime factors of $m$ are larger than $C(n)$, all the obstruction groups vanish except for the case $n \equiv 3 \pmod 8$, where the only non-trivial obstruction group is $H^{n}(L_{m}^{n};\pi_{n-1}(F)) \cong \z/2$.  There is a lift $\overline h_{a}$ on the $(n-1)$-skeleton $L^{(n-1)}$.   It's easy to see from the Atiyah-Hirzebruch spectral sequence that $\widetilde{KO}(L_m^n) \to \widetilde{KO}(L^{(n-1)})$ is an isomorphism. Therefore we have a vector bundle $\xi$ over $L_{m}^{n}$ whose restriction to $L^{(n-1)}$ is the vector bundle given by $\overline h_a$. Since $H^{4i}(L_{m}^{n}) \to H^{4i}(L^{(n-1)})$ is an isomorphism for $i=1, \cdots, k$, we have $p_{i}(\xi)=a_{i}$ for $i=1, \cdots, k$.

When $n \not \equiv 1 \pmod 8$,  $g^{*}\xi$ is trivial since $\pi_{n}BO=0$. When $n \equiv 1 \pmod 8$, there is a split surjective map $g^{*} \colon \widetilde{KO}(L_{m}^{n}) \to \widetilde{KO}(S^{n})$, therefore by changing the vector bundle on the top cell of $L_{m}^{n}$ we get a vector bundle $\xi$ such that  $g^{*}\xi$ is trivial.
\end{proof}

Now given $(p_{1}, \cdots, p_{k}) \in (\z/m)^{k}$, let $\nu L_{m}^{n}$ be the stable normal bundle of $L_{m}^{n}$, write 
$$p(\nu L_{m}^{n})\cdot(1+p_{1}+\cdots +p_{k}) = 1+a_{1}+\cdots +a_{k}+ \textrm{higher terms} $$
By Lemma \ref{lem:odd}, there is an $(n+1)$-dimensional vector bundle $\xi$ over $L_{m}^{n}$ with $p_{i}(\xi)=a_{i}$ for $i=1, \cdots , k$. Let $\pi \colon S(\xi) \to L_{m}^{n}$ be the sphere bundle, then $\pi^{*} \colon H^{i}(L_{m}^{n}) \to H^{i}(S(\xi))$ is an isomorphism for $i\le n-1$, and the total Pontrjagin class of $S(\xi)$ is 
$$p(S(\xi))=\pi^{*}(p(L_{m}^{n}) \cdot p(\xi))=1+p_{1}+\cdots +p_{k}+\textrm{higher terms}$$
The universal cover of $S(\xi)$ is $S^{n} \times S^{n}$.  Let $M_{0}=S(\xi) \sharp \frac{g-1}{m}(S^{n} \times S^{n})$, then its universal cover is $\sharp {g}(S^{n} \times S^{n})$ and $p_{i}(M_{0})=p_{i}(S(\xi))=p_{i}$ for $i=1, \cdots, k$.

\subsubsection{$n$ even}
\begin{lem}\label{lem:even}
Let $L_{m}^{n+1}=L_{m}^{n+1}(1,\cdots, 1)$ be the standard lens space, $g \colon S^{n+1} \to L_m^{n+1}$ be the covering map. Given $a_i \in H^{4i}(L_m^{n+1})=\z/m$ for $i=1,\cdots, k$, there is an $n$-dimensional vector bundle $\xi$ over $L_m^{n+1}$ satisfying the following conditions
 \begin{enumerate}
 \item $p_i(\xi)=a_i$ for $i=1, \cdots, k$;
  \item $g^*(\xi)$ is trivial;
 \item the Euler class $e(\xi)=0$ when $n=4k$.
\end{enumerate}
\end{lem}

\begin{proof}
When $n \ne 4k$, by Lemma \ref{lem:odd} there is an $(n+1)$-dimensional vector bundle $\xi'$ over $L_{m}^{n+1}$ with $p_{i}(\xi')=a_{i}$ for $i=1, \cdots ,k$. The Euler class of $\xi'$ is trivial since $n+1$ is odd, therefore $\xi'$ splits off a trivial line bundle and we get an $n$-dimensional vector bundle $\xi$ over $L_{m}^{n+1}$ with the given Pontrjagin classes.

Now consider the case $n=4k$.
Let $\mathrm{Pon} \colon BO(n-1) \to K(\z,4) \times \cdots \times K(\z,4k)=K$ be the map given by the universal Pontrjagin classes $p_1, \cdots, p_k$, denote the homotopy fiber by $F$. For
$a=(a_{1}, \cdots , a_{k})$
let $h_{a} \colon L_{m}^{n+1} \to K(\z, 4) \times \cdots \times K(\z, 4k)$ be the classifying map of $a$, and consider the lifting problem
$$\xymatrix{
& BO(n-1) \ar[d]^{\mathrm{Pon}} \\
L_{m}^{n+1} \ar[ur]^{\overline h_{a}} \ar[r]^{h_{a}\ \ \ \ \ \ \ \ \ \ \ \ } &  K(\z, 4) \times \cdots \times K(\z, 4k)}$$
The obstruction groups $H^{i}(L_{m}^{n+1};\pi_{i-1}(F))=0$ for $i \le n-1$.  There is an exact sequence
$$0 \to \pi_{n}(F) \to \pi_{n}BO(n-1) \to \pi_{n}(K) \to \pi_{n-1}(F) \to \pi_{n-1}BO(n-1) \to 0.$$
For $n=4k$ we have $\pi_{n-1}BO(n-1)=0$ when $n=4, 8$, and $\pi_{n-1}BO(n-1)=\z/2$ when $n \ne 4, 8$. The stabilization map $\pi_{n}BO(n-1) \to \pi_{n}BO$ is surjective when $n\ne 4,8$, and has an image of index $2$ when $n=4,8$ (c.~f.~\cite{Kerv59}). Therefore the obstruction group $H^{n}(L_{m}^{n+1};\pi_{n-1}F)=0$ and we have a lifting $\overline h_{a}$ on the $n$-skeleton $L^{(n)}$. Denote the corresponding vector bundle by $\xi'$ and let $\xi = \xi' \oplus \varepsilon^{1}$, then $e(\xi)=0$. Let $f \colon S^{n} \to L^{(n)}$ be attaching map of the top cell, then $f^*\xi$ is trivial, since a $4k$-dimensional vector bundle over $S^{4k}$ is trivial if and only if both $p_{k}$ and $e$ are $0$.  Hence $\xi$ extends to a vector bundle over $L_{m}^{n+1}$ such that $p_{i}(\xi)=a_{i}$ for $i=1, \cdots, k$. Since $H^n(L_m^{n+1}) \to H^n(L^{(n)})$ is an isomorphism and $e(\xi|L^{(n)})=0$ by construction, we see that $e(\xi)=0$.

Let $g \colon S^{n+1} \to L_{m}^{n+1}$ be the covering map, consider the vector bundle $g^{*}\xi$.  We may change the vector bundle $\xi$ on the top cell by $\alpha \in \pi_{n+1}BO(n)$, this will not change the Pontrjagin classes and the Euler class, since the restriction of the bundle to $L^{(n-1)}$ is not changed.  The corresponding change for $g^{*}\xi$ is $m\cdot \alpha$. Since $m$ is invertible in the $2$-group $\pi_{n+1}BO(n)$ (c.~f.~\cite{Kerv59}), we see that for a certain choice of $\xi$ the pull-back bundle $g^{*}\xi$ is trivial.
\end{proof}

Let $\pi \colon S(\xi) \to L_{m}^{n+1}$ be the sphere bundle. Since $g^*\xi$ is trivial, the universal cover of $S(\xi)$ is $S^{n-1} \times S^{n+1}$. From the Leray-Serre spectral sequence of the fiber bundle $S^{n-1} \to S(\xi) \to L_m^{n+1}$ it's easy to see that $\pi^* \colon H^{4i}(L_m^{n+1}) \to H^{4i}(S(\xi))$ is an isomorphism for $i=1, \cdots, k$. (For the case $n=4k$ we need the fact that $e(\xi)=0$.) Then as in the $n$ odd case we may realize any given $p_{i}=p_{i}(S(\xi))$ ($i=1, \cdots, k$) by choosing appropriate $\xi$.
Let $S^{n-1} \times D^{n+1} \subset S(\xi)$ be the tubular neighborhood of a fiber $S^{n-1}$. Do surgery as in the proof of Theorem \ref{thm:existence},  the result manifold $N$ has fundamental $\z/m$. The universal cover of $N$ is the result of equivariant surgeries on $\cup_{i=1}^m S^{n-1} \times D^{n+1} \subset S^{n-1} \times S^{n+1}$, hence is $\sharp_{m-1}(S^{n} \times S^{n})$.  Let $W$ be the trace of the surgery, then there are homotopy equivalences
$$S(\xi) \cup e^{n} \simeq W \simeq N \cup e^{n+1}$$
and isomorphisms $H^{4i}(S(\xi)) \cong H^{4i}(W) \cong H^{4i}(N)$ for $i = 1, \cdots, k$, (except for the case $n=4k$, where $H^{n}(W) \to H^{n}(N)$ is injective.) 
The Pontrjagin classes of $W$ restrict to the Pontrjagin classes of $S(\xi)$ and $N$ respectively. Therefore $p_i(N)=p_i(S(\xi))=p_{i}$ under the identification of the cohomology groups with $\z/m$. Let $M_0=N \sharp \frac{g+1-m}{m}(S^n \times S^n)$, then its universal cover is $\sharp g(S^n \times S^n)$ and $p_i(M_0)=p_i(N)=p_{i}$ for $i=1, \cdots, k$.

\subsection{The normal $(n-1)$-type}
In this subsection we determine the normal $(n-1)$-type of the quotient manifold $M$ of a free $\z/m$-action. Identify $H^{4i}(M)$ with $H^{4i}(K(\z/m,1))$, let
$$\mathrm{Pon}_{M} \colon K(\z/m,1) \to K(\z, 4) \times \cdots \times K(\z,4k)$$
be the map induced by $p_1(M), \cdots ,  p_k(M)$, then there is a lifting $\gamma$
$$\xymatrix{
& BSO \ar[d]^{\mathrm{Pon}} \\
K(\z/m,1) \ar[ur]^{\gamma} \ar[r]^{\mathrm{Pon}_{M} \ \ \ \ \ \ \ \ } & K(\z, 4) \times \cdots \times K(\z,4k)}$$
since all the obstruction groups are $0$. Denote the corresponding vector bundle over $K(\z/m,1)$ also by $\gamma$.

Let $p_0 \colon BO\langle n+1 \rangle \to BO$ be the $n$-connected cover of $BO$, $f \colon M \to K(\z/m,1)$ be the classifying map of $\pi_1(M)=\z/m$. By the construction of $\gamma$ we have $p_i(f^*\gamma \oplus \nu M)=0$ for $i=1, \cdots , k$.

\begin{lem}\label{lem:type}
The stable vector bundle $f^*\gamma \oplus \nu M$ has an $O\langle n+1 \rangle$-structure.
\end{lem}
\begin{proof}
Let $\varphi \colon BSO \to X_{n}$ be the $n$-stage of the Postnikov tower of $BSO$, then there is a fibration
$$BO\langle n+1 \rangle \to BSO \stackrel{\varphi}{\rightarrow} X_{n}.$$
Let $g \colon M \to BSO$ be the classifying map of $f^*\gamma \oplus \nu M$, then $g$ has a lifting to $BO\langle n+1 \rangle$ if and only if $\varphi \circ g$ is null-homotopic. Note that $X_n$ is an infinite loop space, we show that  $\varphi \circ g$ is null-homotopic by looking at its image in the the localization $[M, X_n]_{(p)}$ at every prime $p$. If $p$ is not a prime factor of $m$,  the localization at $p$, $[M, X_{n}]_{(p)} \to [\widetilde M, X_{n}]_{(p)}$  is injective. But the pull-back bundle of $f^{*}\gamma \oplus \nu M$ to $\widetilde M$ is the stable normal bundle of $\sharp g(S^{n} \times S^{n})$, hence has an $O\langle n+1 \rangle $-structure.  If $p$ is a prime factor of $m$, then $p > C(n)$, the $n$-stage Postnikov tower of $BSO_{(p)}$ is
$$\mathrm{Pon} \colon BSO_{(p)} \to K(\z_{(p)},4) \times \cdots \times K(\z_{(p)}, 4k).$$
Since $p_i(f^*\gamma \oplus \nu M)=0$ for $i=1, \cdots, k$, we see that $\varphi \circ g$ is null-homotopic in $[M,  K(\z_{(p)},4) \times \cdots \times K(\z_{(p)}, 4k)]$.
\end{proof}
Let $p \colon B = K(\z/m,1) \times BO\langle n+1 \rangle \stackrel{(-\gamma) \times p_0}{\longrightarrow} BSO \times BSO \stackrel{\oplus}{\longrightarrow} BSO$, $\overline g \colon M \to BO\langle n+1 \rangle$ be an $O\langle n+1 \rangle$-structure of $f^{*}\gamma \oplus \nu M$,
then from the construction we have a commutative diagram
$$\xymatrix{
& K(\z/m,1) \times BO\langle n+1 \rangle \ar[d]^{p} \\
M \ar[ur]^{f \times \overline g} \ar[r]^{\nu} & BSO}$$
i.~e.~$\bar \nu =f \times \overline g$ is a lifting of the normal Gauss map $\nu$. It's easy to see that  $\bar \nu$  is $n$-connected and $p$ is $n$-co-connected. Therefore the normal $(n-1)$-type of $M$ is $(B, p)$.

\subsection{Surgery obstruction}\label{subsec:obeven}
From the previous subsection we see that $M$ has the same normal $(n-1)$-type with some model manifold $M_{0}$ if they have the same Pontrjagin classes $p_{i}$ for $i \le k$. In the next subsection we will show that they are bordant in this normal $(n-1)$-type. Parallel to \S \ref{sec:ob}, in this subsection we show that there is a bordism $(W, \bar \nu)$ such that the surgery obstruction $\theta(W, \bar \nu)$ is elementary.

\subsubsection{$n$ odd}
The surgery obstruction is in $l_{4q+3}(\z/m,+)$, or in $l_{15}^{\sim}(\z/m,+)$ when $n=7$ (in this case it's easy to check $\langle w_{8}(B), \pi_{8}(B) \rangle \ne 0$). The $l$-monoid is $4$-periodic, therefore the second case is treated  in Section \ref{sec:ob}. We consider the first case, in which the quadratic refinement $\mu$ takes value in $\Lambda/\langle x+ \bar x \rangle$.  We need a version of Lemma \ref{lem:s}. The proof there goes through, the only difference is now $\mu(\alpha e_{2} + \beta f_{2})=[\alpha \bar \beta] \in \Lambda/\langle x + \bar x \rangle$ can be either $0$ or $[1]$, and $\mu (v e_{2} + s f_{2})=[v \bar s]= [\varepsilon_{2}(v)] \in  \Lambda/\langle x + \bar x \rangle$, where $\varepsilon_{2}$ is the mod $2$ augmentation. Recall from the proof of Lemma \ref{lem:module} that $u=1+g+g^{2}+\cdots +g^{l-1}$ for some $l \in \mathbb N$ coprime to $m$. Therefore there exist $a, b \in \mathbb Z$ such that $bl + am =1$. Then $v=1+g^{l }+ \cdots + g^{(b-1)l}$ satisfies $uv + as =1$, and $\varepsilon (v) =b$. Now since $m$ is odd, we may take $b'=b+m$, which has opposite parity with $b$. Therefore we may always choose an appropriate $v$ such that $\mu(v e_{2} + s f_{2})=\mu(\alpha e_{2} + \beta f_{2})$. The rest goes as in the proof of Lemma \ref{lem:s}.

\subsubsection{$n$ even}
The analysis of the surgery obstruction is similar to the $n$ odd case. First of all in this case we still have $L_{2n+1}(\z/m) = L_{2n+1}^{s}(\z/m)=0$ (\cite[p.227]{HT99}).
Recall from the construction of the model manifolds $M_0$ in the end of \S \ref{subsec:model}, we have $$\mathrm{Ker}(\pi_{n}(M_0) \to \pi_n(B))/\mathrm{rad}=\pi_n(M_0) = \pi_n(N)  \perp H_{+}^{r}(\Lambda)$$
and $\pi_n(N) \cong I^2$, where $I \subset \Lambda$ is the augmentation ideal. The quadratic form $(\pi_n(N), \lambda, \mu)$ is as follows: let $x, y$ be generators of $\pi_n(N)$, where $x$ is represented by the complementary sphere of the surgery, then we have
\begin{equation*}
\mu(x)=0, \ \ \mu(y)=\zeta \in \Lambda/\langle a -                     \bar a \rangle \textrm{ for some $\zeta$}
\end{equation*}
\begin{equation}\label{eqn:quad}
\lambda(x,x)=0, \ \ \lambda(x,y)=1-g, \ \ \lambda(y,y)=\zeta + \bar \zeta
\end{equation}

\begin{lem}
$\varepsilon(\lambda(y,y))=0$
\end{lem}
\begin{proof}
The augmented ideal $I$ is generated by $1-g$. For any $\Lambda$-mod homomorphism $\psi \colon I \to \Lambda$, assume $\psi(1-g)=a$. Then $\varepsilon (\psi(s \cdot (1-g)))=s \cdot \varepsilon(a)=0$, since $s \cdot (1-g)=0$. Therefore $\varepsilon (a)=0$. Apply this to $\psi(-)=\lambda(-,y)$.
\end{proof}

Let $\Lambda^2 \to \pi_{n}(N)$ be the canonical projection, with basis $v_1 \mapsto x$, $v_2 \mapsto y$. Then the induced quadratic form $(\lambda,
\mu)$ on $\Lambda^2$ has the form of (\ref{eqn:quad}). Now consider an isometric embedding $(\Lambda^2, \lambda, \mu)=S \hookrightarrow H_+(\Lambda)^2$ as a half rank direct summand. In view of Lemma \ref{lem:lag} we need to show that $S$ has a Lagrangian complement. Since $\mu(v_1)=0$, after an ambient isometry, we may assume that
$$v_1=e_1, \ \ v_2=a_1 \cdot e_1 + a_2 \cdot e_2 + (1-g)\cdot f_1 + b_2 \cdot f_2$$
As before, since $a_2 \cdot e_2 + (1-g)\cdot f_1 + b_2 \cdot f_2$ is a primitive element, we must have $a_2 \cdot \Lambda + b_2 \cdot \Lambda + I = \Lambda$. On the other hand the fact $\varepsilon(\lambda(v_2,v_2))=0$ implies $\varepsilon(a_2)\cdot \varepsilon (b_2)=0$. Without loss of generalizty, assume $\varepsilon (b_2)=0$, then $\varepsilon (a_2)=1$ and $a_2=1+a(1-g)$ for some $a \in \Lambda$. Let $w_1=a \cdot e_2 + f_1$, $w_2=-\bar a \cdot e_1 + f_2$. The matrix
$$\left ( \begin{array}{cccc}
1 & a_1 & 0 & -\bar a\\
0 & 1+a(1-g) & a & 0\\
0 & 1-g & 1 & 0 \\
0 & b_2 & 0 & 1\\
\end{array} \right )$$
has determinant $1$, therefore $w_1$ and $w_2$ generate a submodule $U$, which is a Lagrangian complement of $S$.

\subsection{Normal bordism class and the proof of Theorem \ref{thm:gen}}
\begin{prop}
There is a surjective homomorphism $\Omega_{2n}^{O\langle n+1\rangle} \to \Omega_{2n}(B,p) $.
\end{prop}
\begin{proof}
The bordism group $\Omega_{2n}(B,p)$ is identified with the twisted bordism group $\Omega_{2n}^{O\langle n +1 \rangle }(K(\z/m,1); \gamma)$.
 The $E_2$-terms in the Atiyah-Hirzebruch spectral sequence are
$$E_2^{p,q}=H_p(K(\z/m,1);\Omega_q^{O\langle n+1 \rangle})$$
The bordism groups $\Omega_q^{O\langle n+1 \rangle}$ are as follows:
\begin{enumerate}
\item When $q \le n-1$, there is an isomorphism $\pi_q^S \cong \Omega_q^{\mathrm{fr}} \stackrel{\cong}{\rightarrow} \Omega_q^{O\langle n+1 \rangle}$.
\item When $q=n$, there is a surjective homomorphism  $\pi_q^S \cong \Omega_q^{\mathrm{fr}} \to \Omega_q^{O\langle n+1 \rangle}$.
\item When $n+1 \le q \le 2n-1$, by surgery, for $q$ odd, there is an isomorphism $\mathrm{Coker}\,J_q \stackrel{\cong}{\rightarrow} \Omega_q^{O\langle n+1 \rangle}$;  for $q \equiv 2 \pmod 4$, there is an injective homomorphism  $\mathrm{Coker}\,J_q \to \Omega_q^{O\langle n+1 \rangle}$ whose image has index at most $2$; for $q \equiv 0 \pmod 4$, there is an isomorphism $\mathrm{Coker}\,J_q \stackrel{\cong}{\rightarrow} \mathrm{Tors}\,\Omega_q^{O\langle n+1 \rangle}$, where $J_q \colon \pi_q(O) \to \pi_q^S$ is the stable $J$-homomorphism, and $\mathrm{Tors}\,\Omega_q^{O\langle n+1 \rangle}$ denotes the torsion subgroup of $\Omega_q^{O\langle n+1 \rangle}$.\end{enumerate}
Therefore if the prime factors of $m$ are all larger than $C(n)$,  all terms in the line $p+q=2n$ are zero, except for the $E_{2}^{0,2n}$-term, which can be identified with $\Omega_{2n}^{O\langle n+1\rangle}$.
\end{proof}

\begin{proof}[Proof of Theorem \ref{thm:gen}]
By surgery, elements in $\Omega_{2n}^{O\langle n+1\rangle}$ are represented by $(n-1)$-connected $2n$-manifolds with an $O\langle n+1\rangle$-structure. Furthermore there are homomorphisms
$$\mathrm{sign} \colon \Omega_{2n}^{O\langle n+1\rangle} \to  \z \ \ \ (\textrm{$n$ even}), \ \ \ \mathrm{Arf} \colon  \Omega_{2n}^{O\langle n+1\rangle} \to \z /2 \ \ \ (\textrm{$n$ odd})$$
such that  elements in the kernels are represented by homotopy spheres $\Sigma \in \Theta_{2n}$.

Therefore if the quotient manifold $M$ has the same normal $(n-1)$-type as certain model manifold $M_{0}$, their difference in $\Omega_{2n}(B,p)$ is represented by an $(n-1)$-connected $2n$-manifold $Y$ with an $O\langle n+1 \rangle$-structure. When $n$ is even, the signature of $Y$ is $0$, therefore there exists a $\Sigma \in \Theta_{2n}$ such that $M \sharp \Sigma$ is $B$-bordant to $M_{0}$. By \S \ref{subsec:obeven} the surgery obstruction is elementary, therefore $M \sharp \Sigma$ is diffeomorphic to $M_0$. The same applies to the case $n$ odd and there is no Kervaire invariant $1$ closed manifold (i.~e.~$n \ne 15$, $31$ and possiblly $63$). In the remaining dimensions, if the Kervaire invariant of $Y$ is $1$, then by surgery there is a smooth Kervaire manifold $K$ (i.~e.~an $(n-1)$-connected almost parallelizable $2n$-manifold with Kervaire invariant $1$ and $H_n(K) \cong \z^2$), such that $M_{0}$ is $B$-bordant to $M \sharp K$. Then $M_{0} \sharp (S^{n} \times S^{n})$ and $M \sharp K$ are $B$-bordant and have equal Euler characteristic, thus $M_{0} \sharp (S^{n} \times S^{n})$ is diffeomorphic to $M \sharp K$. Passing to the universal cover,  $\sharp_{g+m}(S^{n} \times S^{n})$ is diffeomorphic to $\sharp_{g}(S^{n} \times S^{n}) \sharp m \cdot K$. But when $m$ is odd, these two manifolds have different Kervaire invariants, a contradiction.  Now we have shown that $M\sharp \Sigma$  is diffeomorphic to $M_{0}$ for some $\Sigma$. Passing to the universal cover, $\sharp {g}(S^{n} \times S^{n})$ is diffeomorphic to $\sharp {g}(S^{n} \times S^{n}) \sharp m \cdot \Sigma$. Since $\sharp g(S^{n} \times S^{n})$ is the boundary of a parallelizable manifold, its inertia group is trivial, therefore $m  \cdot \Sigma = 0 \in \Theta_{2n}$. When the prime factors of $m$ are larger than $C(n)$, $m$ is invertible in $\Theta_{2n}$, hence $\Sigma$ is the standard sphere, and $M$ is diffeomorphic to $M_{0}$.
\end{proof}

\section{Appendix: computation of $\Omega_{6}(B,p)$}
\begin{prop}
The normal bordism group $\Omega_{6}(B,p)=0$.
\end{prop}

\begin{proof}
We first prove the spin case, where $\Omega_{6}(B,p)=\omsp_6(K(\z/m,1))$. The $E^2$-terms in the Atiyah-Hirzebruch spectral sequence are $E^2_{p,q}=H_p(K(\z/m,1);\omsp_q)$.

When $m$ is odd, it is easy to see that all $E^2$-terms in the line $p+q=6$ are $0$. Therefore $\omsp_6(K(\z/m,1))=0$.

When $m$ is even, the situation is more complicated. We first look at the case $m=2$. It's known that $\omsp_6(\rp^{\infty}) \cong \Omega_5^{\mathrm{Pin^-}}=0$ \cite[Theorem 1 and Corollary 2]{KirTay}.  Nevertheless we examine the Atiyah-Hirzebruch spectral sequence to get information about the differentials, which will be used in the analysis of the spectral sequence for general even $m$. Now the non-zero $E^2$-terms in the line $p+q=6$ are $E^2_{4,2}$ and $E^2_{5,1}$, both isomorphic to $\z/2$. It's known that the differentials $d_2 \colon E^2_{p,1} \to E^2_{p-2, 2}$ are dual to the Steenrod squares $\mathrm{Sq}^2$ (c.~f.~\cite{Teichner}). For $K(\z/2,1)=\rp^{\infty}$ we have $\mathrm{Sq}^2 \colon H^3(\rp^{\infty};\z/2) \to H^5(\rp^{\infty};\z/2)$ is an isomorphism and $\mathrm{Sq}^2 \colon H^4(\rp^{\infty};\z/2) \to H^6(\rp^{\infty};\z/2)$ is trivial. Therefore $E^3_{4,2} \cong \z/2$ is the only non-trivial term in the line $p+q=6$ on the $E^3$-page. Further possibly non-trivial differentials ending at or starting from the position $(4,2)$ are $d_3: E^3_{7,0} \to E^3_{4,2}$ and $d_3: E^3_{4,2} \to E^3_{1,4}$, and we know one of them must be non-trivial since $\omsp_{6}(\rp^{\infty})=0$. Since the edge homomorphism $\omsp_7(\rp^{\infty}) \to H_7(\rp^{\infty})$ is surjective (notice that $\rp^7$ is spin), $E^3_{7,0}$ must survive to infinity, and thus $d_3: E^3_{7,0} \to E^3_{4,2}$ is trivial. Therefore we see that $d_3: E^3_{4,2} \to E^3_{1,4}$ is an isomorphism.

When $m>2$ is even, the nontrivial $E^2$-terms in the line $p+q=6$ are still $E^2_{4,2}$ and $E^2_{5,1}$, both isomorphic to $\z/2$. The mod $2$ cohomology ring and Steenrod operations of $K(\z/m,1)$ are known as follows \cite{Serre53}:
$$H^*(K(\z/m,1);\z/2) \cong \left\{ \begin{array}{ll}
\z/2[x], \  |x|=1 & \textrm{when $m \equiv 2 \pmod 4$} \\
& \\
\z/2[x,y]/(x^2), \ |x|=1, \ |y|=2 & \textrm{when $m \equiv 0 \pmod 4$}
\end{array} \right.
$$
with $\mathrm{Sq}^1y=0$.
From this it's easy to see that the Steenrod square
$$\mathrm{Sq}^2 \colon H^3(K(\z/m,1);\z/2) \to H^5(K(\z/m,1);\z/2)$$
 is an isomorphism. Hence $E^3_{4,2} \cong \z/2$ is the only non-trivial term in the line $p+q=6$ on the $E^3$-page. To show that $d_3 \colon E^3_{4,2} \to E^3_{1,4} \cong \z/m$ is non-trivial, consider the map $\rp^{\infty} \to K(\z/m, 1)$ induced by the non-trivial homomorphism $\z/2 \to \z/m$. From the induced map between the spectral sequences we obtain a commutative diagram
$$\xymatrix{
E^3_{4,2}(\rp^{\infty}) \ar[r]^{\cong \ \ \ \ } \ar[d]^{d_3'} & E^3_{4,2}(K(\z/m,1)) \ar[d]^{d_3} \\
E^3_{1,4}(\rp^{\infty}) \ar[r] & E^3_{1,4}(K(\z/m,1))}
$$
in which the upper horizontal map
$$E^3_{4,2}(\rp^{\infty})=H_4(\rp^{\infty};\z/2) \to H_4(K(\z/m,1);\z/2)=E^3_{4,2}(K(\z/m,1))$$
is an isomorphism,
the lower horizontal map
$$E^3_{1,4}(\rp^{\infty})=H_1(\rp^{\infty}) \to H_1(K(\z/m,1))=E^3_{1,4}(K(\z/m,1))$$
is non-trivial, and $d_3'$ is an isomorphism. Therefore $d_3$ is  nontrivial. This shows that $\omsp_6(K(\z/m,1))=0$.

In the non-spin case  $\Omega_6(B,p)=\omsp_6(K(\z/m,1);\gamma)$. By the Thom isomorphism the $E^2$-terms still can be identified with $H_p(K(\z/m ,1);\omsp_q)$, whereas the differential $d_2 \colon E^2_{p,1} \to E^2_{p-2, 2}$ is dual to
$$\mathrm{Sq}^2+w_2(\gamma)\cup - \colon H^{p-2}(K(\z/m,1);\z/2) \to H^{p}(K(\z/m,1);\z/2)$$
and $d_2 \colon E^2_{p,0} \to E^2_{p-2, 1}$ is a composition
$$H_p(K(\z/m,1)) \to H_p(K(\z/m,1);\z/2) \to H_{p-2}(K(\z/m,1);\z/2)$$
in which the second map is dual to
$$\mathrm{Sq}^2+w_2(\gamma)\cup - \colon H^{p-2}(K(\z/m,1);\z/2) \to H^{p}(K(\z/m,1);\z/2).$$
From the structure of the cohomology ring $H^*(K(\z/m,1);\z/2)$ and its Steenrod operation, it's easy to see
$$\mathrm{Sq}^2+w_2(\gamma)\cup - \colon H^{p-2}(K(\z/m,1);\z/2) \to H^{p}(K(\z/m,1);\z/2)$$
is an isomorphism for $p=6$, $7$.
Therefore the non-trivial $E^2$-terms in the line $p+q=6$, $E^2_{4,2}$ and $E^2_{5,1}$, are killed by $d_2$. This shows that $\omsp_6(K(\z/m,1);\gamma)=0$.
\end{proof}

\noindent \textbf{Acknowledgement.} We would like to thank Matthias Kreck for clarifying the relation between the $L$- and tilde $L$-groups in Lemma \ref{lem:kr}. We would like to thank Ian Hambleton and Diarmuid Crowley for helpful communications.

\bibliographystyle{amsplain}

\end{document}